\documentclass{amsart}
\usepackage{latexsym,amsmath,amsthm,amssymb}%,fullpage}%,bbm}
\newtheorem{theorem}{Theorem}[section]

\newtheorem{proposition}[theorem]{Proposition}%[section]
\theoremstyle{remark}
\newtheorem{remark}[theorem]{Remark}%[section]

\theoremstyle{definition}

\newtheorem{definition}[theorem]{Definition}

\begin{document}

\title{Remarks on the Afriat's theorem and the Monge--Kantorovich problem}

\begin{abstract}
The  famous Afriat's theorem from the theory of revealed preferences establishes 
necessary and sufficient conditions for existence of utility function for a given set of choices and prices.
The result on existence of a {\it homogeneous} utility function  can be considered as a particular fact of the Monge--Kantorovich mass transportation theory.
In this paper we explain this viewpoint and discuss some related questions.
\end{abstract}

\author{Alexander~V.~Kolesnikov}
\address{ Higher School of Economics, Moscow,  Russia}
\email{Sascha77@mail.ru}

\vspace{5mm}
\author{Olga~V. Kudryavtseva}
\address{
Moscow State University, Departments of Economics, Moscow, Russia
}
\email{kudryavtseva@msu.econ.ru}

\vspace{5mm}
\author{Tigran Nagapetyan}
\address{
Fraunhofer ITWM, Kaiserslautern, Germany
}
\email{tigran.nagapetyan@itwm.fhg.de}


\thanks{
This study was carried out within ''The National Research University Higher School of Economics`` Academic Fund Program in 2012-2013, research grant No. 11-01-0175. The first author was supported by RFBR projects 10-01-00518, 11-01-90421-Ukr-f-a, and the program SFB 701 at the University of Bielefeld.
The second author was  supported by  RFBR project 12-06-33030a.
We thank Professor Atshushi Kajii for his interest and stimulating discussions.
}


\maketitle

%\vspace{5mm}

\section{Afriat's theorem}

The first description of the concept of the revealed preferences  can be find in the  work of Samuelson \cite{Sam},
where he presented the weak axiom of revealed preferences.  The strong axiom of revealed preferences (SARP)
was introduced by Houthakker \cite{Hout}. 
It was shown by Afriat  \cite{Afriat} that SARP is a necessary and sufficient condition for  existence of 
an appropriate utility function
for a finite set of choices 
and prices observed (this is  called the rationalization of the preferences relations).
Later Varian \cite{Varian82}, \cite{Varian83} extended the method of \cite{Afriat} by providing  tests for homothetic and additive separability and rationalizing models of behavior.

The connection between the Afriat's theorem and the Monge--Kantorovich problem is known (see, for instance, \cite{Levin}, \cite{Shananin}, \cite{Shananin2} for the connection with the so-called 
"Monge--Kantorovich optimal transshipment problem"). One can also mention the shortest path problem, which is known to be related to the Afriat's theorem 
after Varian \cite{Varian82}, \cite{Varian83}.
This  problem has a solution under assumption 
of absence of negative cycles, which in turn can be viewed as a ''cyclical monotonicity'' assumption.  See chapter 9 in \cite{AMJ} for the description of the shortest path problem is terms of
the linear programing duality.
See also \cite{Rochet}, where the relation with the Rockafellar's cyclical monotonicity theorem is discussed.
Nevertheless, the authors find that
 an instructive and short description of this relation  is somehow missing in the literature. We fill this gap and, applying some recent results on the Monge--Kantorovich problem, give a complete
characterization of the rationalizable data sets from the "transportational" viewpoint. Some related results
based on duality, linear programing, cyclical monotonicity  etc. were obtained in \cite{FST}, \cite{BC}, \cite{Kann}, \cite{Diew}, \cite{FKM}.
See also \cite{Ekeland} for another variational interpretation of the Afriat's theorem. 
For an account in the Monge--Kantorovich problem the reader is referred to \cite{BoKo}, \cite{Vill}.

In the standard model we have $m$ different goods and $n$ observations represented by vectors $X^i \in \mathbb{R}^m_{+}$,
 $1 \le i \le n$
$$
X^i = (x^i_1, \cdots, x^i_m)
$$
with corresponding vectors of prices 
$$P^i =  (p^i_1, \cdots, p^i_m) \in \mathbb{R}_{+}^m
$$
This means that the quantity $x^i_k$ of the  $k$-th good was bought at the price $p^i_k$. 
Thus the total amount of money spent by the i-th customer equals to
$$
\langle X^i, P^i \rangle = \sum_{j=1}^m x^i_j p^i_j.
$$
\begin{remark}
Just for the sake of simplicity we deal with the space $\mathbb{R}_{+}^m$ ($\mathbb{R}_{+} = (0, +\infty)$)
of vectors with {\bf positive} coordinates (zero price and zero consumed amount  of any  good is prohibited).
\end{remark}

A general tool of many classical models in economics is the so-called utility function $u : \mathbb{R}_{+}^m\to \mathbb{R}$.
Given an  utility function $u$ we say that  a customer  prefers $X \in \mathbb{R}_{+}^m$   to $Y \in \mathbb{R}_{+}^m$
iff 
$$
u(X) \ge  u(Y).
$$ 
\begin{remark} It is a standard and natural assumption in the utility function theory that $u$ is homogeneous:
$$u(t X) = t u(X), \ \  X \in \mathbb{R}_{+}^m, \ \ t \in \mathbb{R}_+.$$ In our paper 
the utility functions we deal with are always homogeneous (except of Section 3!).
\end{remark}

Under which assumptions on a given data set there exists a utility function 
 that is consistent with this set of observations (choices)? This was the problem solved by Afriat.
Let us describe a systematical approach based on natural modeling of the customer's behavior. 
We always assume that given a fixed price vector $P^i$ the customer always choose
the most preferable combination of goods $X^i$, i.e. 
$u$ attains its maximal value  on the set 
$\{ Y: \langle Y, P^i \rangle \le \langle X^i, P^i \rangle, \ \ Y \in \mathbb{R}^{m}_{+}\}$.

\begin{definition}
\label{main-def}
We say that the set $\{(X^i,P^i), \ 1 \le i \le n\}$ admits an utility function $u$ (or $u$ rationalizes this set) if  $u(Y) < u(X^i)$ for every $ Y \in \mathbb{R}^{m}_{+}$ satisfying
$ \langle Y, P^i \rangle < \langle X^i, P^i \rangle$.
\end{definition}

\begin{remark}
Let $u$ be continuous. Then this definition has a simple geometrical meaning:  every hyperplane $\{ Y : \langle P^i, Y - X^i \rangle =0\}$ is supporting to the set $\{Y: u(Y) \ge u(X^i)\}$.
\end{remark}

Necessary and sufficient condition for existence of $u$ for a given data set
was obtained in \cite{Afriat} (see Theorem \ref{Af-dis} below).

\section{Monge--Kantorovich problem}

\begin{remark}
In contrary to the previous section, we denote below the finite sets in $\mathbb{R}^m \times \mathbb{R}^m$ by
$(x_i,y_i)$ instead of $(X^i,P^i)$.
\end{remark}

In the modern formulation of the Monge--Kantorovich problem one considers a couple of  probability measures $\mu$ and $\nu$ on $\mathbb{R}^m$
and a {\bf cost function} $c(x,y)$. 

\begin{definition}
 Denote by $P_{\mu,\nu}$  the set of probability measures on  $X \times Y = \mathbb{R}^m \times \mathbb{R}^m$ satisfying
$$
\mbox{Pr}_X  P = \mu, \ \ \mbox{Pr}_Y P = \nu.
$$
\end{definition}
Here $\mbox{Pr}_X P$, $\mbox{Pr}_Y P$ are projections of $P$ onto $X$, $Y$ respectively, i.e. measures defined by
$$
\mbox{Pr}_X(A) = P(A \times Y), \ \mbox{Pr}_Y(B) = P(X \times B).
$$

The measure $P$ on $\mathbb{R}^m \times \mathbb{R}^m$ solves the Monge--Kantorovich problem if it satisfies  the following properties
\begin{itemize}
\item[1)] $P \in P_{\mu,\nu}$
\item[2)] $P$ is the minimum of the functional $K(P) = \int c(x,y) \ dP$.
\end{itemize}

Interpreting $c(x,y)$ as a {\bf transportation cost} of some production unit from the point $x$ to the point $y$, the integral  $\int c(x,y) \ dP$
equals to the {\bf total cost} of transportation. The measures $\mu$ and $\nu$ are initial and final distribution of the total production respectively.

We give the following  example. 
Let $P$ be the uniform distribution on a discrete set $\{(x_i,y_i), \ 1 \le i \le n\}$,
i.e. $P((x_i,y_i))=\frac{1}{n}$.
We consider  $P$ to be a candidate to solve the Monge--Kantorovich problem for $\mu = Pr_X$, $\nu = Pr_Y$. 
The total cost equals to
$$
\frac{1}{n} \sum_{i=1}^{n} c(x_i,y_i).
$$
Now let $\sigma \in S_m$ be any {\bf permutation} of indices. Take a measure $P_{\sigma}$ which is 
the uniform distribution on the set $$
S_{\sigma} = \{(x_i,y_{\sigma(i)})\}, \ 1 \le i \le n\}.
$$
Note that $P_{\sigma}$ still has the same projections. The new total transportation cost equals to
$\frac{1}{n}\sum_{i=1}^{n} c(x_i,y_{\sigma(i)})$.
Thus, a necessary condition for being optimal in the Monge--Kantorovich sense is the following inequality between total costs:
\begin{equation}
\label{c-monot}
\sum_{i=1}^{n} c(x_i,y_i) \le \sum_{i=1}^{n} c(x_i,y_{\sigma(i)}).
\end{equation}
This observation leads  to the following definition
\begin{definition}
A set $A \subset X \times Y$ is called $c$-monotone if
 every finite subset $\{(x_i,y_i), \ 1 \le i \le n\} \subset A$ and
every permutation $\sigma \in S_n$ satisfies
(\ref{c-monot}).
\end{definition}
It is well-known that every permutation $\sigma$ can be decomposed into a product of several {\bf cyclical} permutations, i.e. permutations of the type
$$
\sigma(i_1) = i_2, \sigma(i_2)=i_3, \cdots, \sigma(i_{k-1})=i_k, \sigma(i_k)=i_1.
$$ 
This immediately gives us that $c$-monotonicity is equivalent to  {\bf $c$-cyclical monotonicity}.

\begin{definition}
A set $A \subset X \times Y$ is called $c$-cyclically monotone if
 every finite subset $\{(x_i,y_i), \ 1 \le i \le n\} \subset A$
 satisfies
\begin{equation}
\label{cycl-monot}
\sum_{i=1}^{n} c(x_i,y_i) \le \sum_{i=1}^{n} c(x_i,y_{i+1})
\end{equation}
with the argeement $y_{n}=y_1$.
\end{definition}

\begin{definition}
We say that $(x,y) \in \mathbb{R}^m \times \mathbb{R}^m$ belongs to {\bf $c$-superdifferential} of a function $u: \mathbb{R}^m \to \mathbb{R}$
if
$$
u(z) \le u(x) + c(z,y)-c(x,y)
$$
for every $z \in \mathbb{R}^m$.
\end{definition}

The  following theorem gives a full characterization of the solutions to the Monge--Kantorovich problem. 

\begin{theorem}
\label{solut-MK}
Let $X=Y$ be a complete, separable, metric space, $\mu$ and $\nu$ be Borel probability measures thereon.
Assume that $c(x,y) : X \times Y \to [0,+\infty)$ is a lower semi-continuous nonnegative cost function and $\int \int c(x,y) \ d \mu(x) \ d \nu(y) < \infty$. Let $\pi$ be a Borel probability measure on $X \times Y$ such that $Pr_X \pi = \mu$, $Pr_Y \pi = \nu$.
Then the following statements are equivalent
\begin{itemize}
\item[1)] $\pi$ is the solution to the Monge--Kantorovich problem with  marginals $\mu$, $\nu$ and the cost function $c$
\item[2)] there exists a $c$-cyclically monotone set $\Gamma$ satisfying $\pi(\Gamma)=1$
\item[3)] there exists a function $u$ and a  set $\Gamma$ satisfying $\pi(\Gamma)=1$ such that $\Gamma$ is contained in the $c$-superdifferential of $u$.
\end{itemize}
\end{theorem}

\begin{remark}
It is easy to check that the theorem holds also for $c$ uniformly bounded from below: $c(x,y) \ge K$, $K \in \mathbb{R}$.
We will  use the theorem in the case of {\bf continuous} cost functions only. 
\end{remark}

The facts collected in this theorem are the cornerstones of the Monge--Kantorovich theory.
The equivalence of 2) and 3) for $c(x,y)=(x-y)^2$ was proved by Rockafellar (see \cite{Roca}). The relation 3) $\Longrightarrow$ 2)
is elementary. Indeed, one has for $(x_i, y_i) \in \Gamma$
$$
u(x_{i+1}) - u(x_i) \le c(x_{i+1},y_i) - c(x_i,y_i).
$$
Summing up these inequalities one obtains the desired cyclical monotonicity property. The equivalence  of 2) and 3) for
general $c$ was obtained by R{\"u}schendorf \cite{Ruesch}. The implication 1) $\Longrightarrow$ 2) is very well-known and was apparently discovered   for the first time by Knott and Smith \cite{KS} for $c=(x-y)^2$. The implication  
2) $\Longrightarrow$ 1) is relatively recent and was proved in sufficient generality in \cite{SchT}.

\begin{remark}
\label{c-concave}
One can always assume that the function $u$ from the item 3) of Theorem \ref{solut-MK} is defined on the whole $\mathbb{R}^m$ and is $c$-concave, i.e. there exists a function $\varphi(y)$ such that $u(x) = \inf_{y \in \mathbb{R}^m}(c(x,y)-\varphi(y))$.
\end{remark}

Let us discuss connections of the Monge--Kantorovich theory with the revealed preferences.
The key observation here is the following.

\begin{proposition}
\label{MK-A}
The set of data $\{(X^i,P^i)\}$ admits a positive homogeneous utility function $u$  if and only if
it is contained in the $c$-superdifferential of the function $v=\log u$  for $c(x,y) = \ln \langle x, y \rangle$. 
\end{proposition}
\begin{proof}
Note that the relation 
\begin{equation} \label{rel1} \langle P^i, Z \rangle < \langle P^i, X^i \rangle \Longrightarrow u(Z) < u(X^i)
\end{equation} for a positive homogeneous function $u$ is equivalent to the following inequality 
\begin{equation}
\label{rel2}
\frac{u(X^i)}{\langle P^i, X^i \rangle}  \ge \frac{u(Z)}{\langle P^i, Z \rangle}.
\end{equation}
Indeed, (\ref{rel1})  follows immediately from  (\ref{rel2}). 

Assume that (\ref{rel1}) holds. Take any $X^i$ and $Z$  and find positive $\lambda$ such that $\langle P^i, \lambda \cdot Z \rangle = \langle P^i, X^i \rangle$. Then  it  follows from 
(\ref{rel1}) that $u(\lambda' \cdot Z) < u(X^i)$ for every $\lambda'<\lambda$. Using that $u$ is homogeneous and $\lambda = \frac{\langle P^i, X^i \rangle}{\langle P^i,  Z \rangle}$ we immediately get  (\ref{rel2}).

We finish the proof with the observation that (\ref{rel2}) is equivalent to the inequality 
$$v(X^i) - v(Z) \ge \log \langle P^i, X^i \rangle - \log \langle P^i, Z \rangle$$
for $v=\log u$ and every $Z$. The proof is complete.
\end{proof}

We are almost ready to get the Afriat's theorem from Theorem \ref{solut-MK}. To this end we identify the
set  $\{ X^i\}$, $ 1 \le i \le n$ with the probability measure
$$
\mu = \frac{1}{n} \delta_{X^i},
$$
where $\delta_{X_i}$ is the Dirac measure concentrated in $X_i$. Similarly
$$
\nu = \frac{1}{n} \delta_{P^i}.
$$
Finally, let $\pi$ be a measure on $\mathbb{R}^{m}_{+} \times \mathbb{R}^{m}_{+} $ defined by
$$
\pi = \frac{1}{n} \delta_{(X^i ,P^i)}.
$$

\begin{theorem} \label{Af-dis}{\bf (Generalized Afriat's theorem, discrete case)}
Let $(X^i,P^i) \subset \mathbb{R}^{m}_{+} \times \mathbb{R}^{m}_{+}$ be a finite set, $c(x,y) = \ln \langle x, y \rangle$. The following statements are equivalent
\begin{itemize}
\item[1)] $\pi$ is the solution to the Monge--Kantorovich problem with  marginals $\mu$, $\nu$ and the cost function $c$
\item[2)] the set $(X^i, P^i)$ is $c$-cyclically monotone
\item[3)] the set $(X^i, P^i)$ admits a  positive homogeneous utility function $u$.
\end{itemize}
\end{theorem}
\begin{proof}
By Proposition \ref{MK-A}  3) is equivalent to the property that the set $(X^i, P^i)$
is included to the $c$-superdifferential of $v=\log u$. Hence the statement is a particular case of Theorem \ref{solut-MK}.
\end{proof}

\begin{remark}
The cyclical monotonicity for $c(x,y) = \ln \langle x, y \rangle$ (property 2) is equivalent to the following inequality
for any  $k$ different indexes $i_1, i_2, \cdots, i_k $ 
\begin{equation}
\label{sapr2}
\langle P^{i_1}, X^{i_1} \rangle \cdot \langle P^{i_2}, X^{i_2} \rangle  \cdots \langle P^{i_k}, X^{i_k} \rangle 
\le 
\langle P^{i_1}, X^{i_2} \rangle \cdot \langle P^{i_2}, X^{i_3} \rangle  \cdots \langle P^{i_k}, X^{i_1} \rangle. 
\end{equation}
The latter is known as  a {\bf homogeneous axiom of revealed preferences} (HARP).

\begin{remark}
In the homogeneous case HARP is equivalent to SARP (see, for instance,  \cite{Varian}).
\end{remark}

\end{remark}

\section{Monge--Kantorovich problem in the non-homogeneous case}

 One can ask the following natural  question. Assume we are given a {\bf non-homogenious} rationalizable discrete data $(X^i,P^i)$.  Whether exists a cost function $c(x,y)$ such that
the corresponding utility function $u$ is a potential for some Monge--Kantorovich with the cost function $c(x,y)$?
The answer is yes, but $c$ highly depends  on the data set in general (unlike the homogeneous case where one can always set $c = \ln\langle x, y \rangle$).
Indeed, it is known that for every rationalizable $(X^i,P^i)$ there exist positive numbers $s_i$ such that the following system
of linear inequalities has a solution
\begin{equation}
\label{ys-ij}
y_j - y_i \le s_i \langle P^i, X^j-X^i \rangle.
\end{equation}
This is the most difficult step in the proof of the general Afriat's theorem (see \cite{FST} for relatively short arguments).

We set
$$
u(X^i) = y_i,\ \ c(X,P^i) = s_i \langle P^i,X \rangle.
$$
One can extend $u$ to $\mathbb{R}^m_+$:
$$
u(x) = \min_{1 \le i \le m} \{y_i + s_i \langle P^i, x - X^i \rangle \}.
$$
Clearly, (\ref{ys-ij}) means that $(X^i,P^i)$ is  included in the $c$-superdifferential of $u$. By Theorem \ref{solut-MK} the data set $(X^i,Y^i)$ 
is a support of a measure $\pi$ solving some optimal transportation problem for the cost function $c$.

\section{Continuous case and optimal transportation}

In this section we deal only with the cost function $c(x,y) = \ln \langle x, y \rangle$.

Theorem \ref{Af-dis} has a natural generalization to the non-discrete case. Consider a non-finite (even non-countable)  data
of observations $$D = \{(x_i,y_i) \subset \mathbb{R}^m_{+} \times \mathbb{R}^m_{+} ,  i \in I\}.$$
As we have seen in the previous section, it is convenient to deal with {\bf probability measures} on $D$. Thus we  assume that a probability measure $\pi$ on $S$ is given. All the statements below are formulated up to a set of zero measure.
The projection of $\pi$ are denoted by $\mu$ and $\nu$ respectively.

Just for the technical reasons and for the sake of simplicity we will assume in this section the following:

{\bf Assumption:} There exists a compact set $K \subset \mathbb{R}^m_{+} \times \mathbb{R}^m_{+}$ such that $\pi(K)=1$.

\begin{remark}
Under this assumption the cost function $c(x,y)$ is continuous on the support of $\pi$. This makes applicable all the theorems from  the previous section.
\end{remark}

\begin{definition}
We say that $\pi$ admits a utility function $u$ if and only if for $\pi$-almost all $(x_i,y_i)$ and every $z \subset \mathbb{R}^{m}_+$
one has $$
u(z) < u(x_i) 
$$
provided $\langle x_i, z \rangle < \langle x_i, y_i \rangle$.
\end{definition}

The following result is just the continuous version of Theorem \ref{Af-dis} and the proof follows the same arguments.

\newpage

\begin{theorem} \label{Af-cont}{\bf (Generalized Afriat's theorem, continuous case)}
Let $c(x,y) = \ln \langle x, y \rangle$. The following statements are equivalent
\begin{itemize}
\item[1)] $\pi$ is the solution to the Monge--Kantorovich problem with  marginals $\mu$, $\nu$ and the cost function $c$
\item[2)] there exists a $c$-cyclically monotone set $\Gamma$ satisfying $\pi(\Gamma)=1$
\item[3)] $\pi$ admits a  positive homogeneous utility function $u$.
\end{itemize}
\end{theorem}

Let us make an important remark on the structure of the optimal solutions. 
Let $S=S^{m-1} = \{x \in \mathbb{R}^m_{+}: \|x\|=1\}$ be the $m-1$-dimensional sphere of radius $1$. Let $P_S$ be the projection on $S^{m-1}$:
$$
P_S(x) = \frac{x}{|x|} \in \mathbb{R}^{m}_{+}.
$$ 
In the same way we set
$$
P_S(y) = \frac{y}{|y|} \in \mathbb{R}^{m}_{+}.
$$
 We denote by $\mu_S = \mu \circ P^{-1}_S$ the projection of $\mu$ onto $S$, i.e. a measure on $S$ which is defined by the formula
$$
\mu_S(A) = \mu(P_{S}^{-1}(A)).
$$
Here  $A \subset S$ is an arbitrary Borel set and $P^{-1}_S(A) = \{z: P_S(z) \in A\}$ is the preimage of $A$ under $P_S$.
In the same way we define $\nu_S$ and $$\pi_{S \times S} = \pi \circ (P^{-1}_S(x), P^{-1}_S(y) ).$$

If is clear that given the marginals $\mu$ and $\nu$ the problem of minimizing of $\int \ln \big\langle x, y \big\rangle \ d \pi$ is equivalent to the problem of minimizing of $\int \ln \langle \frac{x}{|x|}, \frac{y}{|y|} \rangle \ d \pi$.
Indeed,  this follows from the relation
$$\int \ln \big\langle \frac{x}{|x|}, \frac{y}{|y|} \big\rangle \ d \pi
=  \int \ln \langle x, y \rangle \ d \pi - \int \log |x| \ d \mu - \int \log |y| \ d \nu
$$
and the fact that the quantities $\ \int \log |x| \ d \mu $, $\int \log |y| \ d \nu$ are fixed.
This means that {\bf  $\pi$ is $c$-optimal if and only if its projection $\pi_{S \times S}$ on $S \times S$ is optimal for the marginals $\mu_S$, $\nu_S$ and the cost function $ \ln \langle x, y \rangle$}.

Now let us assume that $\mu_S$ and $\nu_S$ have densities with respect to the {\bf surface measure $\sigma$} on $S$:
$$
\mu_S = f \cdot \sigma, \ \ \nu_S = g \cdot \sigma.
$$
Then it is well-known (see \cite{Vill}, \cite{BoKo} and the references therein) that there exists a mapping $T: S \to S$ with the following property:
$$
\pi(\Gamma)=1, \ \Gamma = \{(x,T(x)), \ x \in S\}.
$$ 
In particular, $\pi$-almost all points  $(x_i,y_i)$ satisfy the relation $y_i = T(x_i)$ and $\nu_S$ is the image of $\mu_S$ under $T$ in the following sense
$$
\nu_S(T(A)) = \mu_S(A),\ \ \ \mbox{where} \  T(A) = \{y: y=T(x) \ \mbox{for some}\  x \in A  \}
$$
for every Borel set $A \subset S$. The mapping $T$ is called {\bf optimal transportation} mapping.
It can be also identified with the inverse demand function.

It is  easy to understand the relation of  $T$ with the utility function $u$.
If $u$ is differentiable at the point $x_i$ (this fails actually only on a set of $\mu$-measure zero), then the hyperplane $L$ given by the equation $\langle z - x_i, y_i \rangle=0$
touches the level set of $u$ exactly at the point $x_i$. Hence $\frac{\nabla u(x_i)}{|\nabla u(x_i)|}$
is the normal vector of $L$ satisfying $$\frac{\nabla u(x_i)}{|\nabla u(x_i)|} = \frac{y_i}{|y_i|}.$$

{\bf Conclusion:} $\nu_S$ is the image of $\mu_S$ under the mapping $x \to \frac{\nabla u(x)}{|\nabla u(x)|}$, $x \in S$.

We note that $T(x)$ coincides with  the {\bf normal vector} to the surface $\{y : u(y) = u(x) \}$ taken at the point $x$.
It follows from  Remark \ref{c-concave} that this surface can be assumed  convex (meaning that the set  $\{y : u(y) \ge u(x) \}$ is convex).
This provides a relation  with  the so-called Alexandrov's problem.

For a  convex surface $F \subset \mathbb{R}^n$ we consider its normal mapping into the sphere $S$:
$
F \ni x \mapsto N(x),
$
where $N(x)$ is the normal to the tangent plane to $F$ at the point~$x$.
Suppose that the origin is inside of $F$. Then $F$  can be parameterized by means of
a radial function: $F \ni r(x) = \varrho(x) x$, $x \in S$.
Let us define a mapping $T_F\colon S \to S$,
$T_F(x) = N(r(x))$.

\begin{definition}
Let $\mu$ and $\nu$ be a couple of probability  measures $\mu$ and $\nu$ on $S$.
We say that a convex surface $F$ is a solution to the 
{\bf Alexandrov's problem}
if $\nu$ is the image of $\mu$ under $T_F$.
\end{definition}

A generalized version of this problem was posed and solved by A.D.~Alexandrov in \cite{Alex}.  
Rewriting this problem analytically one gets a kind of {\bf Monge--Amp{\'e}re equation} which involves the {\bf Gauss curvature} of $F$.
 It was shown by V. Oliker \cite{Oliker} (see also recent development in \cite{Bert})
that the Monge--Kantorovich problem for the  function $c(x,y)= - \log \langle x, y \rangle$ can be used to construct the solution to the 
Alexandrov's problem.
Note that our situation is almost the same, the only difference is the sign of the cost function. 

\begin{remark}
It is easy to see that the whole theory concerning revealed preferences can be extended in the same way if instead of the standard scalar product one considers any 
function $b(x,y)$ which is homogeneous in both variables: $t b(x,y) = b(tx, y) = b(x, ty), \ t \ge 0$. Namely, given a data set $\{(x_i,y_i)\}$ one tries to find a homogeneous function $u$ with the property
$$
b(x_i, y_i) > b(z,y_i) \Longrightarrow u(z) < u(x_i).
$$
This problem can be reduced to the optimal transportation problem for the cost function $c(x,y)= \log b(x,y)$.
\end{remark}

\section{Negative cycles and potential fields}

Let us consider again a general cost function $c(x,y)$ and a couple of probability measures $\mu$, $\nu$. We assume that $\mu$ and $\nu$ have densities with respect to the Lebesgue measure.
 Let $T$ be the optimal transport of 
$\mu$ onto $\nu$. We will assume that both $c(x,y)$ and $T$ are sufficiently regular. It 
follows from the general Kantorovich duality statement and from Theorem \ref{solut-MK} 3) as well that
$T$ and $u$ are related by the formula
$$
\nabla_x c(x,T(x)) = \nabla u(x)
$$
(see, for instance formula (2.63) in \cite{Vill}).
In particular, $\nabla_x c(x,T(x))$ is a potential vector field and the integral
$$
\int_{\gamma} \nabla_x c(x,T(x)) \ d \gamma = u(\gamma(1)) - u(\gamma(0)) 
$$ 
along any smooth path $\gamma : [0,1] \to \mathbb{R}^d$ depends on
$\gamma(1)$ and $\gamma(0)$ only.

This observation can be interpreted as a continuous analog of the following well-known 
statement from optimizational combinatorics:  a discrete graph admits a {\bf shortest path} for every couple of vertices 
 if and only if  it has no {\bf negative cycles}. Given a finite data $(x_i,y_i)$ let us endow every edge $(x_i, x_j)$ of
the directed graph
with vertices  $(x_i,x_j)$, $1 \le i,j \le n$, with the ''distance'' $a_{ij} = c(x_j,y_i) - c(x_i,y_i)$ 
(the number $a_{ij}$ is allowed to be negative). The sequence $x_{i_1}, x_{i_2}, \cdots, x_{i_n}$ is called negative cycle if
$$
\sum_{k=1}^{n} a_{i_k i_{k+1}} <0, \ \ i_{n+1} = i_1.
$$
It follows immediately from the definition that the absence of the negative cycles is equivalent to the $c$-cyclical monotonicity of the data set. In the absence of negative cycles every two vertices admit a shortest path joining them.
A classical computational algorithm for finding the shortest path based on dynamical programing principle is the {\bf Warshall--Floyd} algorithm. 

Now let us assume that we have a continuous $c$-cyclically monotone date set $D = \{(x, p(x)), \ x \in X \subset \mathbb{R}^d \}$, where  $p: X \to \mathbb{R}^d$ is a sufficiently regular mapping, and a smooth path $\gamma : [0,1] \to X$
with $\gamma(0)=\gamma(1)$. Let us pick numbers
$x_{i} = \gamma(i/n), \ 1 \le i \le n$. By the $c$-cyclical monotonicity
$$
0 \le \frac{1}{n} \sum_{i=1}^{n} \bigl( c(x_{i+1}, y_i) -  c(x_{i}, y_i)\bigr).
$$ 
Passing to the limit $n \to \infty$ we get
$\int_{\gamma} \nabla_x c(x,p(x)) d \gamma \ge 0.$  Running the cycle in the opposite direction we get in the same way 
$-\int_{\gamma} \nabla_x c(x,p(x)) d \gamma \ge 0.$
Finally we get a continuous version of the ''absence of negative cycles'' principle:
if $D$ is $c$-cyclically monotone, then
$$
\int_{\gamma} \nabla_x c(x,p(x)) d \gamma = 0.
$$
Thus every ''continuous cycle'' is zero in the smooth setting.
Hence
$$
\nabla_x c(x,p(x)) = \nabla u
$$
for some potential $u$. Clearly, all the paths joining two points $x_0,x_1$ has the same ''length'' $\int_{\gamma} \nabla_x c(x,p(x)) d \gamma$.

In particular, we get for $c = \log \langle x, y \rangle$ that
$$
\frac{p(x)}{\langle x, p(x) \rangle} = \nabla u(x).
$$
This relation has been  studied systematically in \cite{Shananin}, \cite{Shananin2} from the viewpoint
of the theory of index numbers.

\end{document}